\documentclass[12pt,a4paper]{article}
\usepackage{amsmath,amssymb,amsthm,mathtools}
\usepackage[colorlinks=true,linkcolor=blue,urlcolor=blue,citecolor=blue]{hyperref}

\newtheorem{theorem}{Theorem}[section]
\newtheorem{lemma}[theorem]{Lemma}

\newtheorem{conjecture}[theorem]{Conjecture}
\theoremstyle{remark}\newtheorem{remark}[theorem]{Remark}

\DeclareMathOperator{\cdim}{cd}

\title{Finite covers of a product of surfaces \\
with  bounded rank and arbitrarily large systole}
\author{Koji Fujiwara
\thanks{supported by JSPS Grant-in-Aid for Scientific Research 25H00588}
}
\date{}

\begin{document}\maketitle

\begin{abstract}
We construct finite covers of a fixed product of two closed hyperbolic surfaces of genus two whose fundamental groups are generated by at most fifteen elements and whose injectivity radii tend to infinity. The construction uses fibre products over finite groups.
 These covers are closed aspherical four-manifolds with universal cover $\mathbb H^2\times\mathbb H^2$. 
 They give counterexamples to a conjecture of Avramidi and Delzant for symmetric spaces of higher rank, in the case \(\mathbb H^2\times\mathbb H^2\).
\end{abstract}

\section{Introduction}

For an isometric action by $G$ on a metric space $Y$, we define the {\it minimal displacement} by
\[
 \operatorname{disp}_Y(G)=\inf\{d_Y(y,gy):y\in Y,\ 1\ne g\in G\}.
\]
We define the {\it translation length} of $g$ by   $\ell_Y(g)=\inf_y d_Y(y,gy)$.
We denote by \(d(G)\) the {\it rank} of \(G\), that is, the minimal number of generators of \(G\).

Avramidi and Delzant proved the following theorem. 
\begin{theorem}[Avramidi--Delzant {\cite[Corollary~34]{AD1}}]
Suppose that $G$  acts on a $\delta$-hyperbolic space $Y$ isometrically, and 
$d(G) \le n$. If
\[
\operatorname{disp}_Y(G)>
100\delta\log_2((n+1)!),
\]
then $G$ is free.
\end{theorem}

Sufficiently large displacement in a hyperbolic space forces the strong algebraic conclusion that the group is free.

The theorem applies to rank-one symmetric spaces, which are $\delta$-hyperbolic.
They conjecture a higher-rank analogue \cite[Conjecture~4]{AD2}.
 We state the following group-theoretic
consequence of their conjecture.
 Indeed, as explained in \cite{AD2}, applying Conjecture 4 there to the
augmentation ideal of $K[G]$ gives the stated bound on
$\operatorname{cd}_K G$.

\begin{conjecture}[Avramidi--Delzant, group-theoretic form]\label{conj-AD} Let $X$ be a symmetric space with curvature $-1 \le \kappa \le 0$  of real rank $r$. For every $n$ there exists a constant $f(n)$ with the following property:
Let $G$ be a group with $d(G) \le n$.
Suppose $G$ acts on $X$ isometrically such that  \[ \operatorname{disp}_X(G)>f(n). \] Then, 
$\operatorname{cd}_K G\le r $ for every field $K$.
\end{conjecture}

 They specifically mention $X=\mathbb H^2\times\mathbb H^2$, which is a symmetric space of real rank two with curvature $-1 \le \kappa  \le 0$. We give a counterexample in this case.

\bigskip

\begin{theorem}\label{main}
Let $S$ be a closed oriented hyperbolic surface of genus two, with curvature $-1$. Let \(A\) and \(B\) be two copies of the fundamental group of \(S\). 
% We put the product metric on $X$. 
% The sectional curvatures of $X$ lie in $[-1,0]$.
Let
$
X=\mathbb H^2\times\mathbb H^2
$ with the product metric.
$A\times B$ acts on $X$ isometrically such that $S\times S=(A\times B)\backslash X$. 

For every $R>0$ there exists a finite-index subgroup $H_R<A\times B$ such that
\[
 d(H_R)\le15,\qquad \operatorname{disp}_{X}(H_R) >R
 %\ell_X(h)>R\quad\text{for every }1\ne h\in H_R.
\]
The corresponding connected finite cover
$M_R=H_R\backslash X$ of $S\times S$
is a closed oriented aspherical four-manifold, and
$
\operatorname{cd}_K H_R=4
$
for every field $K$.

\end{theorem}

Since \(X\) is a Hadamard manifold,
\[
\operatorname{inj}(M_R)
=\frac12\operatorname{sys}(M_R)
=\frac12\operatorname{disp}_X(H_R)
>R/2.
\]
Thus these covers have injectivity radii tending to infinity.

\bigskip
The construction extends to any fixed number of surface factors; see
Theorem~\ref{prop:higher-products}. For every fixed \(q\ge 2\), it gives finite covers
\(M_R^{(q)}\) of \(S^q\) whose injectivity radii tend to infinity while
the ranks of their fundamental groups remain uniformly bounded.

The construction of $H_R$ is a fibre product.
Given epimorphisms $\alpha:A\twoheadrightarrow Q$ and
$\beta:B\twoheadrightarrow Q$, the {\it fibre product} is 
defined by
\[
A\times_Q B
 =\{(a,b)\in A\times B:\alpha(a)=\beta(b)\}.
\]
If $Q$ is finite, then $A\times_Q B$ has index $|Q|$ in $A\times B$.

For a given \(R>0\), let \(E_A\subset A\setminus\{1\}\) and \(E_B\subset B\setminus\{1\}\) contain one representative of each nontrivial conjugacy class with translation length at most \(R\). These sets are finite.
The following lemma gives an algebraic formulation of what we need.

\begin{lemma}[Finite-quotient lemma]\label{finite-input}
Let $A,B$ be genus-two surface groups, and let
$E_A\subset A\setminus\{1\}$ and $E_B\subset B\setminus\{1\}$ be finite.
There exist a finite group $Q$ and epimorphisms
\[
\alpha:A\twoheadrightarrow Q,\qquad \beta:B\twoheadrightarrow Q
\]
with the following two properties:
\begin{enumerate}
\item For every $u\in E_A\cup\{1\}$ and $v\in E_B\cup\{1\}$ with $(u,v)\neq(1,1)$, the elements $\alpha(u)$ and $\beta(v)$ are not conjugate in $Q$.
\item The group $Q$ admits a finite presentation with at most three relators.
\end{enumerate}
\end{lemma}
Applying the lemma to these sets and putting \(H_R=A\times_Q B\), property (1) excludes nontrivial elements of \(H_R\) with translation length at most \(R\), while property (2) gives a uniform bound on \(d(H_R)\).
The specific bound three is not important; any uniform bound would suffice. The proof is given in Section~\ref{finiteproof}.

We also record in the appendix a sharper general bound for the rank of a fibre product, which may be useful elsewhere.

\bigskip

\section{Generator bounds of fibre products}\label{construction}
%\subsection{Finite quotient with conjugacy separation}

We record two elementary lemmas that will be used to bound the number of generators of the fibre products appearing in the construction.

\begin{lemma}[Generator bound for a fibre product]\label{fiber-rank}
Let $\alpha:A\twoheadrightarrow Q$ and $\beta:B\twoheadrightarrow Q$ be epimorphisms. Suppose $A$ and $B$ are generated by at most $a$ and $b$ elements, respectively, and $\ker\alpha$ is normally generated in $A$ by at most $c$ elements. Then
\[
 d(A\times_Q B)\le a+b+c.
\]
\end{lemma}
This lemma is essentially the same 
as \cite[Lemma~2.2(a)]{Mi}. Since his notation 
is slightly different, we give an explanation. 
\begin{proof}
 In the notation of \cite[Lemma~2.2(a)]{Mi}, the fibre product
$A\times_Q B$ is a subdirect product of $A\times B$, since both
$\alpha$ and $\beta$ are epimorphisms. Moreover,
\[
A\cap(A\times_Q B)=\ker\alpha,
\]
where $A$ is identified with $A\times\{1\}$.

Thus Minasyan's subgroup $N_1$ is  $\ker\alpha$ in our
setting. The proof of \cite[Lemma~2.2(a)]{Mi} explicitly constructs a
generating set for the subdirect product consisting of at most
$a+b+c$ elements. Hence,
$
d(A\times_Q B)\le a+b+c.
$

\end{proof}

\begin{lemma}[Changing generating sets]\label{marking}
Suppose a group $Q$ has a presentation with $t$ generators and $c$ relators. 
For any finite generating set \(\{x_1,\ldots,x_s\}\) of \(Q\), it has a presentation on that set with at most \(c+s\) relators.
\end{lemma}
\begin{proof}
This is a standard consequence of Tietze transformations. To be precise, write \(Q=\langle y_1,\ldots,y_t\mid R_1,\ldots,R_c\rangle\). Choose words \(w_i(x)\) representing \(y_i\) and \(v_j(y)\) representing \(x_j\). Add the generators \(x_j\) with defining relations \(x_j=v_j(y)\). The relations \(y_i=w_i(x)\) are then consequences, so we may add them and then eliminate the \(y_i\). This gives a presentation with \(c+s\) relators:
$$ Q=\left\langle x_1,\ldots,x_s\ \middle|\ R_i(w(x))\ (1\le i\le c),\ x_j=v_j(w(x))\ (1\le j\le s) \right\rangle. $$
\end{proof}

\section{Proof of Theorem \ref{main}}

We now prove Theorem~\ref{main}, assuming Lemma~\ref{finite-input}.
The proof of Lemma~\ref{finite-input} is given in Section~\ref{finiteproof}.

\begin{proof}
For $(a,b)\in A\times B$, the product metric on $X$ gives
\begin{equation*}
 \ell_X(a,b)^2=\ell_{\mathbb H^2}(a)^2+\ell_{\mathbb H^2}(b)^2.
\end{equation*}

Fix $R>0$.
Since \(S\) is compact, each of \(A\) and \(B\) has only finitely many conjugacy classes of nonidentity elements with translation length at most \(R\).
Choose finite sets \(E_A,E_B\) of representatives of all conjugacy classes of nonidentity elements with translation length at most \(R\).

Apply  Lemma~\ref{finite-input} to \(E_A\) and \(E_B\), obtaining \(Q,\alpha,\beta\),   and put
$$H_R=A\times_{Q}B.$$ 

We first show that
$
d(H_R)\le 15.
$
Since $S$ has genus $2$, let $a_1,\ldots,a_4$ be the standard
generators of $A$. Then
$
\{\alpha(a_1),\ldots,\alpha(a_4)\}
$
is a generating set of $Q$. By Lemma~\ref{finite-input}(2), $Q$ has a
presentation with at most three relators. 
Therefore, by Lemma~\ref{marking}, $Q$ has a presentation with
generators $\alpha(a_1),\ldots,\alpha(a_4)$ and relators
$R_1,\ldots,R_m$, where $m\le 7$.

Let
$
F_4=\langle x_1,\ldots,x_4\rangle
$
be the free group on four generators, and define
\[
\varphi:F_4\twoheadrightarrow Q,
\qquad
\varphi(x_i)=\alpha(a_i).
\]
Regarding $R_1,\ldots,R_m$ as words in $F_4$, we have
\[
\ker\varphi
=\langle\!\langle R_1,\ldots,R_m\rangle\!\rangle_{F_4}.
\]

Now define
\[
\pi:F_4\twoheadrightarrow A,
\qquad
\pi(x_i)=a_i.
\]
Then
$
\varphi=\alpha\pi.
$
Since $\pi$ is surjective and $\varphi=\alpha\pi$, we have
$
\pi(\ker\varphi)=\ker\alpha.
$

Since $\ker\varphi$ is normally generated in $F_4$ by
$R_1,\ldots,R_m$, it follows that $\ker\alpha$ is normally generated
in $A$ by
$
\pi(R_1),\ldots,\pi(R_m).
$
Thus $\ker\alpha <A$ is normally generated by at most seven elements in $A$.

Finally, since both $A$ and $B$ are generated by four elements,
Lemma~\ref{fiber-rank} gives
$
d(H_R)\le 4+4+7=15.
$

\bigskip

We show 
 that if $1\neq(a,b)\in H_R$, then $\ell_X(a,b)> R$.
Suppose not.
Then,  $\ell_{\mathbb H^2}(a) \le R$
and $\ell_{\mathbb H^2}(b) \le R$. 
So, choose $u\in E_A\cup\{1\}$ conjugate to $a$ in $A$; and
$v\in E_B\cup\{1\}$ conjugate to $b$ in $B$.
Write $a=sus^{-1}$ and $b=tvt^{-1}$. 
Since $(a,b) \in H_R$, $\alpha(a)=\beta(b)$.
This gives
\[
\alpha(u)=\bigl(\alpha(s)^{-1}\beta(t)\bigr)\,
\beta(v)\,\bigl(\alpha(s)^{-1}\beta(t)\bigr)^{-1}.
\]
Thus $\alpha(u)$ and $\beta(v)$ are conjugate in $Q$. 
By the property (1) in \ref{finite-input}, 
we have $u=v=1$.
But this implies $a=b=1$, contradicting $(a,b)\not=1$.

Since \(M_R=H_R\backslash X\) is compact, its systole is attained and
$
\operatorname{sys}(M_R)=\operatorname{disp}_X(H_R).
$
Thus, 
$
\operatorname{sys}(M_R)=\operatorname{disp}_X(H_R)>R.
$

The quotient \(M_R\) is a finite cover of \(S\times S\), hence a closed aspherical \(4\)-manifold with universal cover \(X=\mathbb H^2\times\mathbb H^2\).
Since \(M_R\) is a closed aspherical \(4\)-manifold, $\cdim_K H_R=4$.
This completes the proof. 
\end{proof}
We note that \(\ker\alpha\times\ker\beta\) has finite index in \(H_R\), so \(H_R\) is reducible.

\section{Proof of Lemma \ref{finite-input}}\label{finiteproof}

The proof uses the finite groups \(\mathrm{PSL}_2(\mathbb F_p)\).

\begin{proof}
Let
$
Q=\operatorname{PSL}_2(\mathbb F_p),
$
where the prime $p>3$ will be chosen below.

As a consequence of the full residual freeness of surface groups, going back to
Baumslag (see \cite[Corollary~2.2]{BGSS}), for every genus-two surface group
$\Gamma$ and every finite subset $E\subset \Gamma\setminus\{1\}$ there is an
epimorphism,
$
r:\Gamma\twoheadrightarrow F_2
$,
which is nontrivial on every element of $E$.

Applying this to $A$ and $B$, we obtain epimorphisms
\[
r_A:A\twoheadrightarrow F_2,
\qquad
r_B:B\twoheadrightarrow F_2
\]
such that
\[
r_A(u)\ne1\quad(u\in E_A),
\qquad
r_B(v)\ne1\quad(v\in E_B).
\]

Write $F_2=\langle x_1,x_2\rangle$. For a word $w\in F_2$ and
${\mathbf X}=(X_1,X_2)\in\mathrm{SL}_2(\mathbb R)^2$, define
\[
T_w(\mathbf X)=\operatorname{tr}(w(X_1,X_2))^2.
\]
The function $T_w$ is a polynomial in the matrix entries of $X_1$ and $X_2$, with integer
coefficients. Indeed, for matrices of determinant one, inversion is
polynomial in the matrix entries.

If $w\ne1$, then $T_w$ is not identically equal to $4$. To see this,
choose a Schottky pair $X_1,X_2\in\mathrm{SL}_2(\mathbb R)$. Then every
nontrivial word $w(X_1,X_2)$ projects to a hyperbolic element of
$\mathrm{PSL}_2(\mathbb R)$, and hence $T_w(\mathbf X)>4$.

For each $u\in E_A\cup\{1\}$ and $v\in E_B\cup\{1\}$ with
$(u,v)\ne(1,1)$, define
\[
F_{u,v}(\mathbf X,\mathbf Y)=T_{r_A(u)}(\mathbf X)-T_{r_B(v)}(\mathbf Y),
\]
where $\mathbf X=(X_1,X_2)$ and $\mathbf Y=(Y_1,Y_2)$ are in 
$\operatorname{SL}_2(\mathbb R)^2$.
For the identity word, $T_1(\mathbf X)=4$ for every $\mathbf X$.
The polynomials \(F_{u,v}\) will be used to rule out conjugacy by comparing squared traces.

Each $F_{u,v}$ is a nonzero polynomial function in the matrix entries
of $X_1,X_2,Y_1,Y_2$, with integer coefficients.
Indeed, suppose first that $u\ne1$. Choose $\mathbf X=(X_1,X_2)$ to be a Schottky pair
and put $\mathbf Y=(I,I)$. Since $r_A(u)\ne1$, we have
$
T_{r_A(u)}(\mathbf X)>4,
$ by the previous paragraph, while  $T_{r_B(v)}(\mathbf Y)=4$. Hence
$
F_{u,v}(\mathbf X,\mathbf Y)\ne0.
$
If $u=1$, then $v\ne1$, and the same argument with
$\mathbf X=(I,I)$ and $\mathbf Y$ a Schottky pair gives
$F_{u,v}(\mathbf X,\mathbf Y)\ne0$.

\bigskip

We next show that, for all sufficiently large primes $p$, one can
choose four matrices $X_1,X_2,Y_1,Y_2$ in $\mathrm{SL}_2(\mathbb F_p)$ for which
$
F_{u,v}(\mathbf X,\mathbf Y)\ne0
$
for every $u\in E_A\cup\{1\}$ and $v\in E_B\cup\{1\}$ with
$(u,v)\ne(1,1)$.

On the set of matrices in $\mathrm{SL}_2$ whose upper-left entry is
nonzero, write
\[
M(x,y,z)=
\begin{pmatrix}
x&y\\
z&(1+yz)/x
\end{pmatrix},
\qquad x\ne0.
\]
Thus four such matrices are parametrized by twelve variables. 
To make all these nonvanishing conditions hold simultaneously, take the product of the finitely many polynomials \(F_{u,v}\), and substitute the above expressions into this product.

Since this chart is a dense open subset of $\mathrm{SL}_2^4$, none
of the $F_{u,v}$ is identically zero on the chart. Hence their
product is not identically zero there. After substitution, this product is a rational
function whose denominator is a monomial in the four upper-left
entries. Since these entries are nonzero on the chart, multiplying by
the denominator does not change the zero set there. We therefore
obtain a nonzero polynomial
\[
N\in\mathbb Z[z_1,\ldots,z_{12}].
\]

Let $D$ be the total degree of $N$. For all but finitely many primes
$p$, the reduction $\overline N$ of $N$ modulo $p$ is nonzero.
For every such prime $p$, the polynomial $\overline N$ has at most
$
Dp^{11}
$
zeros in $\mathbb F_p^{12}$.

Now,
$
|\mathrm{SL}_2(\mathbb F_p)|=p(p^2-1),
$
and the proportion of matrices outside the above chart is $1/(p+1)$.
Indeed, if
$
\begin{pmatrix}
a & b\\
c & d
\end{pmatrix}
\in \operatorname{SL}_2(\mathbb F_p)
$
has $a=0$, then $-bc=1$. Hence $b$ has $p-1$ possible values,
$c$ is then uniquely determined, and $d$ is arbitrary. Thus there are
$p(p-1)$ such matrices.

Therefore, if four matrices are chosen independently and uniformly from
$\mathrm{SL}_2(\mathbb F_p)$, the probability that at least one of them
lies outside the chart,  or that the resulting value of $\overline N$
is zero, is at most
\[
\frac{4}{p+1}
+
\frac{Dp^{11}}{|\mathrm{SL}_2(\mathbb F_p)|^4}.
\]
This tends to zero as $p\to\infty$, since
$|\mathrm{SL}_2(\mathbb F_p)|=p(p^2-1)$.
It follows that the probability that
$
F_{u,v}(\mathbf X,\mathbf Y)\ne 0
$
for every $u\in E_A\cup\{1\}$ and $v\in E_B\cup\{1\}$ with
$(u,v)\ne(1,1)$ tends to one as $p\to\infty$.

\bigskip

By \cite[Theorem~1.1]{EV}, the probability that a random pair of
elements of $\operatorname{SL}_2(\mathbb F_p)$ generates the group
tends to one as $p\to\infty$. This result was proved earlier by Kantor and Lubotzky \cite{KL}.
Hence the probability that both pairs
$\mathbf X$ and $\mathbf Y$ generate $\operatorname{SL}_2(\mathbb F_p)$ also tends to one.

Therefore, for all sufficiently large primes $p$, there exist pairs
$
\mathbf X=(X_1,X_2)$ and $\mathbf Y=(Y_1,Y_2)
$
in $\mathrm{SL}_2(\mathbb F_p)^2$ such that
\[
\langle X_1,X_2\rangle=\mathrm{SL}_2(\mathbb F_p),
\qquad
\langle Y_1,Y_2 \rangle=\mathrm{SL}_2(\mathbb F_p),
\]
and
$
F_{u,v}(\mathbf X,\mathbf Y)\ne0
$
for every $u\in E_A\cup\{1\}$ and $v\in E_B\cup\{1\}$ with
$(u,v)\ne(1,1)$.

Fix such a prime $p$ and such pairs $\mathbf X$ and $\mathbf Y$.
We define epimorphisms
\[
e_{\mathbf X},e_{\mathbf Y}:F_2\twoheadrightarrow \mathrm{SL}_2(\mathbb F_p)
\]
by
\[
e_{\mathbf X}(x_i)=X_i,\qquad e_{\mathbf Y}(x_i)=Y_i
\quad (i=1,2).
\]

Let
\[
\rho:\operatorname{SL}_2(\mathbb F_p)\twoheadrightarrow
\operatorname{PSL}_2(\mathbb F_p)
=\operatorname{SL}_2(\mathbb F_p)/\{\pm I\}
\]
be the quotient map.
Define epimorphisms
\[
\alpha=\rho\circ e_{\mathbf X}\circ r_A,
\qquad
\beta=\rho\circ e_{\mathbf Y}\circ r_B.
\]

We now prove property~{\rm(1)}. Let
$u\in E_A\cup\{1\}$ and $v\in E_B\cup\{1\}$, with
$(u,v)\ne(1,1)$. Put
\[
U=e_{\mathbf X}(r_A(u)),\qquad V=e_{\mathbf Y}(r_B(v)).
\]
By the definition of $e_{\mathbf X}$,
\[
r_A(u)(X_1,X_2)=e_{\mathbf X}(r_A(u))=U.
\]
Hence, by the definition of $T_w$,
we have $
T_{r_A(u)}(\mathbf X)=\operatorname{tr}(U)^2.
$
Similarly,
$
T_{r_B(v)}(\mathbf Y)=\operatorname{tr}(V)^2.
$

We show that if  $\alpha(u)$ and $\beta(v)$ are conjugate in
$\operatorname{PSL}_2(\mathbb F_p)$,
then $F_{u,v}(\mathbf X,\mathbf Y)=0$.
We note, as a side remark, that the trace argument below is valid for arbitrary \(u\in A\) and \(v\in B\), and does not use the special choice of \(\mathbf X\) and \(\mathbf Y\).

Indeed, since
$
\operatorname{PSL}_2(\mathbb F_p)
=\operatorname{SL}_2(\mathbb F_p)/\{\pm I\},
$
there is $C\in\operatorname{SL}_2(\mathbb F_p)$ such that
$
U=\pm CVC^{-1}.
$
Hence
$
\operatorname{tr}(U)^2=\operatorname{tr}(V)^2.
$
Therefore, by the identities above, we have
\[
F_{u,v}(\mathbf X,\mathbf Y)
=
\operatorname{tr}(U)^2-\operatorname{tr}(V)^2
 =0
\]

On the other hand, by the choice of $\mathbf X$ and $\mathbf Y$,
$
F_{u,v}(\mathbf X,\mathbf Y)\ne0.
$
This is a contradiction. Thus $\alpha(u)$ and $\beta(v)$ are not
conjugate.  This proves
property~{\rm(1)}.

\bigskip
Finally, by \cite[(1.4)]{Sunday},
$\mathrm{PSL}_2(\mathbb Z/m\mathbb Z)$ admits a presentation with
three relators for every odd $m$. Therefore $\mathrm{PSL}_2(\mathbb F_p)$ has a presentation with three relators. 
This proves property~{\rm(2)} and completes the proof.
\end{proof}

\section{Higher fibre products}
We discuss  the fibre product construction for more than two factors. 
Throughout this section, $S$ is the closed hyperbolic surface of
genus two used above, and products are equipped with the product
metric.

\begin{theorem}\label{prop:higher-products}
For every integer $q\ge2$ and every $R>0$, there exists a
finite-index subgroup
$
        H_R^{(q)}<\pi_1(S)^q
$
such that
\[
 d\bigl(H_R^{(q)}\bigr)\le 11q-7,
 \qquad
 \operatorname{disp}_{(\mathbb H^2)^q}
       \bigl(H_R^{(q)}\bigr)>R.
\]

The corresponding connected finite cover
$
        M_R^{(q)}
        =H_R^{(q)}\backslash(\mathbb H^2)^q
$
of $S^q$ is a closed locally symmetric manifold of dimension
$2q$, with universal cover $(\mathbb H^2)^q$ of real rank $q$.
Also,
$
        \operatorname{inj}\bigl(M_R^{(q)}\bigr)>R/2.
$
\end{theorem}

\begin{proof}
Let $A_1,\ldots,A_q$ be copies of $\pi_1(S)$. For each $i$, let
$E_i\subset A_i\setminus\{1\}$ contain one representative of every
nontrivial conjugacy class of translation length at most $R$.

The proof of Lemma~\ref{finite-input} applies simultaneously
to these $q$ groups. It gives a finite group
$Q=\operatorname{PSL}_2(\mathbb F_p)$ and epimorphisms
\[
        \alpha_i:A_i\twoheadrightarrow Q
        \qquad (1\le i\le q)
\]
with the following property: for every $i\ne j$ and every
\[
 u_i\in E_i\cup\{1\},\qquad
 u_j\in E_j\cup\{1\},\qquad
 (u_i,u_j)\ne(1,1),
\]
the elements $\alpha_i(u_i)$ and $\alpha_j(u_j)$ are not conjugate
in $Q$.

As in the proof of Lemma~\ref{finite-input}, for each
$i$ choose an epimorphism
$
        r_i:A_i\twoheadrightarrow F_2
$
which is nontrivial on every element of $E_i$.

For each $1\le i\le q$, introduce a pair of matrix
variables
\[
 \mathbf X_i=(X_{i,1},X_{i,2})
 \in\operatorname{SL}_2(\mathbb F_p)^2.
\]
For $i\ne j$ and
\[
 u_i\in E_i\cup\{1\},\qquad
 u_j\in E_j\cup\{1\},\qquad
 (u_i,u_j)\ne(1,1),
\]
consider the polynomial
\[
 F_{i,j;u_i,u_j}(\mathbf X_i,\mathbf X_j)
 =
 T_{r_i(u_i)}(\mathbf X_i)
 -
 T_{r_j(u_j)}(\mathbf X_j).
\]
%These are the pairwise squared-trace polynomials used in the proof of Lemma~\ref{finite-input}.

Each of the resulting finitely many polynomials is nonzero, by the
same argument as before.
For each of the $2q$ matrix variables, use the  chart
introduced in the proof of Lemma~\ref{finite-input},
on which the upper-left matrix entry is nonzero. Thus the
$2q$ matrices are parametrized by $6q$ scalar variables, and
the probability that at least one matrix lies outside this
chart is at most
$
        2q/(p+1).
$

Let $D$ be the total degree of the product of these polynomials
after restricting to the charts and clearing denominators,
 and let $\varepsilon_p$ denote the
probability that a random pair in
$\operatorname{SL}_2(\mathbb F_p)$ fails to generate the group.

A union bound shows that the probability that at least one
matrix lies outside the chart, at least one of these polynomials
vanishes, or at least one of the $q$ pairs fails to generate
$\operatorname{SL}_2(\mathbb F_p)$ is at most
\[
 \frac{2q}{p+1}
 +
 \frac{D p^{6q-1}}
      {|\operatorname{SL}_2(\mathbb F_p)|^{2q}}
 +
 q\varepsilon_p.
\]

For fixed $q$, this tends to zero as $p$ tends to infinity,
since
$|\operatorname{SL}_2(\mathbb F_p)|=p(p^2-1)$ and
$\varepsilon_p\to0$ by \cite[Theorem~1.1]{EV}.

Choose a sufficiently large prime $p$ for which this failure
probability is less than one, and choose pairs
\[
        \mathbf X_i=(X_{i,1},X_{i,2})
        \qquad (1\le i\le q)
\]
lying in the chosen charts and satisfying all the required
nonvanishing and generation conditions. Let
$
        e_i:F_2\twoheadrightarrow
        \operatorname{SL}_2(\mathbb F_p)
$
be defined by $e_i(x_j)=X_{i,j}$, let
\[
        \rho:\operatorname{SL}_2(\mathbb F_p)
        \twoheadrightarrow\operatorname{PSL}_2(\mathbb F_p)
\]
be the quotient map, and set
\[
        Q=\operatorname{PSL}_2(\mathbb F_p),
        \qquad
        \alpha_i=\rho\circ e_i\circ r_i.
\]

The group $Q$ admits a presentation with at most three relators.
For each $i$, apply the argument from the proof of
Theorem~\ref{main} to the epimorphism $\alpha_i:A_i\twoheadrightarrow Q$.
Using Lemma~\ref{marking}, we conclude that $\ker\alpha_i$ is normally
generated in $A_i$ by at most seven elements.

\bigskip

Set
\[
 H_R^{(q)}
 =
 \left\{
 (g_1,\ldots,g_q)\in A_1\times\cdots\times A_q:
 \alpha_1(g_1)=\cdots=\alpha_q(g_q)
 \right\}.
\]
This subgroup has finite index $|Q|^{q-1}$.

Suppose that a nonidentity element $(g_1,\ldots,g_q)$ of this
subgroup has translation length at most $R$. Then, each $g_i$ has
translation length at most $R$, so it is conjugate in $A_i$ to
some $u_i\in E_i\cup\{1\}$. The elements $\alpha_i(u_i)$ are all
conjugate in $Q$.

Since $(g_1,\ldots,g_q)$ is nontrivial, there is an index $i$
such that $u_i\ne1$. Choose any $j\ne i$. Then
$(u_i,u_j)\ne(1,1)$, while $\alpha_i(u_i)$ and
$\alpha_j(u_j)$ are conjugate in $Q$. This contradicts the
pairwise nonconjugacy condition.
Therefore, every nonidentity element in $H_R^{(q)}$ has translation length
greater than $R$. Compactness of the quotient gives the asserted
strict bound on displacement.

To bound the number of generators of $H_R^{(q)}$, define, for $1\le j\le q$,
\[
 H^{(j)}
 =
 \left\{
 (g_1,\ldots,g_j):
 \alpha_1(g_1)=\cdots=\alpha_j(g_j)
 \right\}.
\]
Here $H^{(1)}=A_1$.
The homomorphism
\[
 \delta_j:H^{(j)}\longrightarrow Q,\qquad
 \delta_j(g_1,\ldots,g_j)=\alpha_1(g_1)=\cdots=\alpha_j(g_j),
\]
is surjective, since all the maps $\alpha_i$ are surjective.

For $j\ge2$, we may view $H^{(j)}$ as the fibre product
 $A_j\times_Q H^{(j-1)}$ with respect to
\[
        \alpha_j:A_j\twoheadrightarrow Q
        \quad\text{and}\quad
        \delta_{j-1}:H^{(j-1)}\twoheadrightarrow Q.
\]

Applying Lemma~\ref{fiber-rank} with $A=A_j$ and
$B=H^{(j-1)}$, and using the fact that $\ker\alpha_j$ is
normally generated by at most seven elements, we have 
\[
        d(H^{(j)})\le4+d(H^{(j-1)})+7.
\]

Since $d(H^{(1)})=4$, induction yields
\[
        d(H_R^{(q)})\le4+11(q-1)=11q-7.
\]

Since $H_R^{(q)}$ has finite index in $\pi_1(S)^q$, the quotient
$
        M_R^{(q)}
        =H_R^{(q)}\backslash(\mathbb H^2)^q
$
is a connected finite cover of $S^q$. In particular, it is a
closed locally symmetric manifold of dimension $2q$, with
universal cover $(\mathbb H^2)^q$ of real rank $q$. Since
$(\mathbb H^2)^q$ is a Hadamard manifold,
\[
 \operatorname{inj}\bigl(M_R^{(q)}\bigr)
 =
 \frac12\operatorname{disp}_{(\mathbb H^2)^q}
          \bigl(H_R^{(q)}\bigr)
 >
 R/2
\]
\end{proof}

\begin{remark}

For each fixed $q\ge3$, we expect that $M_R^{(q)}$ admits a finite
CW structure with at most $d(H_R^{(q)})$ one-cells, and hence at
most $11q-7$ one-cells, while
$\operatorname{disp}_{(\mathbb H^2)^q}(H_R^{(q)})\to\infty$
as $R\to\infty$.
A possible approach is to use one-handle elimination
(cf.\ \cite[Lemma~6.15]{RS}), since $\dim M_R^{(q)}=2q\ge6$,
and then obtain the required CW decomposition from a Morse--Smale
flow. We do not pursue this here.

In higher rank, these examples suggest
that one may choose finite CW structures with a uniformly bounded number of one-cells, 
while the
displacement tends to infinity.
In contrast, Avramidi and Delzant \cite[Theorem~39(2)]{AD1}
show that if a group $G$ acts isometrically on a $1$-hyperbolic
space $X$ with
$\operatorname{disp}_X(G)>100\log_2((n+1)!)$, where $n$ is a
positive integer, then every aspherical CW complex with fundamental
group $G$ has more than $n$ cells in each dimension
$0<k<\operatorname{cd}(G)$.

\end{remark}

\section*{Appendix: Generator bounds of fibre products}
\label{appendix}

Although we do not need it for the proof of Theorem~\ref{main},
we record a sharper upper bound on $d(A\times_Q B)$ than the bound
 \(2a+b+c\) that follows directly from that proof. The argument is more involved, but the estimate may be useful elsewhere.

\begin{lemma}[abc lemma]
\label{lemma-abc}
Let $\alpha:A\twoheadrightarrow Q$ and
$\beta:B\twoheadrightarrow Q$ be epimorphisms.
Suppose that $d(A)=a$ and $d(B)=b$, and that $Q$ admits a finite
presentation with $c$ relators. Then
\[
d(A\times_Q B)\le a+b+c.
\]
\end{lemma}

The generating set used in the proof can also be checked directly to generate \(G\), but we give a slightly more conceptual argument.

\begin{proof}
Choose epimorphisms
$
  \pi_A:F_A\to A$ and $
  \pi_B:F_B\to B,
$
where
$
  F_A=F(x_1,\ldots,x_a)$ and $
  F_B=F(y_1,\ldots,y_b).
$
Put
\[
  \phi_A=\alpha\circ\pi_A,
  \qquad
  \phi_B=\beta\circ\pi_B.
\]
Let
\[
  G=F_A\times_Q F_B
   =\{(x,y)\in F_A\times F_B:\phi_A(x)=\phi_B(y)\}.
\]

The map \(\pi_A\times\pi_B\) restricts to an epimorphism
$
  G\longrightarrow A\times_Q B.
$
% Indeed, if \((a_0,b_0)\in A\times_Q B\), then any lifts
% \(x\in F_A\) and \(y\in F_B\) of \(a_0\) and \(b_0\), respectively,
% satisfy \(\phi_A(x)=\phi_B(y)\).  
Thus it is enough to prove that
\(G\) is generated by \(a+b+c\) elements.

Fix a presentation
\[
  Q=\langle z_1,\ldots,z_t\mid R_1,\ldots,R_c\rangle.
\]
Choose lifts \(\widehat z_k\in F_A\) with
\(\phi_A(\widehat z_k)=z_k\), for \(1\le k\le t\).  For each \(j\),
choose a word \(w_j(z_1,\ldots,z_t)\) representing \(\phi_B(y_j)\).
Since \(F_B\) is free, the assignments
\[
  u(y_j)=w_j(\widehat z_1,\ldots,\widehat z_t)
\]
define a homomorphism \(u:F_B\to F_A\) satisfying
$
  \phi_A\circ u=\phi_B.
$

For each \(i\), choose \(v_i\in F_B\) such that
\[
  \phi_B(v_i)=\phi_A(x_i).
\]

Let \(L\) be the subgroup generated by the following \(a+b+c\) elements:
\[
  (x_i,v_i) \quad (1\le i\le a),
\]
\[
  (u(y_j),y_j) \quad (1\le j\le b),
\]
\[
  \bigl(R_r(\widehat z_1,\ldots,\widehat z_t),1\bigr)
  \quad (1\le r\le c).
\]
All these elements lie in \(G\). This follows respectively from the
choice of \(v_i\), from \(\phi_A\circ u=\phi_B\), and from the fact that
\(R_r(z_1,\ldots,z_t)=1\) in \(Q\).  Hence \(L< G\).

We show that \(L=G\).  Define
\[
  N=\{x\in F_A:(x,1)\in L\}.
\]
By the definition of $G$, we have \(N\le\ker\phi_A\).  Moreover, the first
projection maps \(L\) onto \(F_A\), because \(L\) contains
\((x_i,v_i)\) for every \(i\).  Consequently, \(N\) is normal in
\(F_A\).  Indeed, given \(g\in F_A\), choose \((g,h)\in L\).  Then, for
\(n\in N\),
\[
  (g,h)(n,1)(g,h)^{-1}=(gng^{-1},1)\in L,
\]
so \(gng^{-1}\in N\).

Set
\[
  \overline F_A=F_A/N.
\]
Since \(N\le\ker\phi_A\), the homomorphism \(\phi_A\) induces an
epimorphism
\[
  \overline\phi_A:\overline F_A\longrightarrow Q.
\]

We will prove \(N=\ker\phi_A\) by showing that \(\overline \phi_A\) is an isomorphism.
For every relator \(R_r\), one of the generators of \(L\) is $(R_r(\hat z_1, \dots, \hat z_t), 1)$.
Hence by the definition of $N$, 
\[
  R_r(\widehat z_1,\ldots,\widehat z_t)\in N.
\]
Thus the elements \(\widehat z_kN\in\overline F_A\) satisfy the relators
\(R_1,\ldots,R_c\) in $\overline F_A$, and therefore define a homomorphism
\[
  s:Q\longrightarrow\overline F_A,
  \qquad
  s(z_k)=\widehat z_kN.
\]
The choice of the lifts gives
$
  \overline\phi_A\circ s=\operatorname{id}_Q.
$
Thus \(s\) is a section of \(\overline\phi_A\). 

We claim that \(s\) is
surjective. First, 
for each generator \(y_j\), we have
\[
\begin{aligned}
u(y_j)N
&=w_j(\widehat z_1,\ldots,\widehat z_t)N =w_j(\widehat z_1N,\ldots,\widehat z_tN) \\
&=s\bigl(w_j(z_1,\ldots,z_t)\bigr) =s(\phi_B(y_j)).
\end{aligned}
\]
Since both sides define homomorphisms from \(F_B\) to \(\overline F_A\), it follows that
$$
u(y)N=s(\phi_B(y))
$$
for every 
$y\in F_B$.

Also, since
$
  y\mapsto (u(y),y)
$
is a homomorphism, the second family of generators of \(L\) implies that
$
  (u(y),y)\in L
  $ for every 
  $y\in F_B.$
Hence, 
\[
  (x_i,v_i)(u(v_i),v_i)^{-1}
   =\bigl(x_i u(v_i)^{-1},1\bigr)\in L.
\]
Thus \(x_i u(v_i)^{-1}\in N\), and hence
\[
  x_iN=u(v_i)N
      =s(\phi_B(v_i))
      =s(\phi_A(x_i)).
\]
Every generator \(x_iN\) of \(\overline F_A\) therefore belongs to the
image of \(s\), so \(s\) is surjective.

Together with
\(\overline\phi_A\circ s=\operatorname{id}_Q\), this shows that \(s\)
and \(\overline\phi_A\) are inverse isomorphisms.  Hence,
$\ker\overline\phi_A=(\ker\phi_A)/N$
is trivial, and therefore \(N=\ker\phi_A\).

Finally, let \((x,y)\in G\).  We have
\[
  (x,y)=\bigl(xu(y)^{-1},1\bigr)(u(y),y).
\]
Here
\[
  \phi_A(x)=\phi_B(y)=\phi_A(u(y)),
\]
so \(xu(y)^{-1}\in\ker\phi_A=N\), and hence
\(\bigl(xu(y)^{-1},1\bigr)\in L\).  The second factor
\((u(y),y)\) also belongs to \(L\), as observed above.  Thus
\((x,y)\in L\), proving \(L=G\).  Thus,
$
  d(A\times_Q B)\le d(G)\le a+b+c.
$
\end{proof}

\medskip
\noindent{\bf Acknowledgements.}
The author  thanks the Mathematical Institute, University of Oxford, for its hospitality during the preparation of this work.
He is grateful to Grigori Avramidi, Thomas Delzant and Ashot Minasyan  for comments.

ChatGPT (GPT-5.6 Sol) was used in developing the proof of Lemma \ref{finite-input} and in polishing the proof of Lemma \ref{lemma-abc}.
 All mathematical content is the responsibility of the author.

\bigskip

\noindent
\textsc{Koji Fujiwara}\\
Okinawa Institute of Science and Technology Graduate University,
Okinawa 904-0495, Japan\\
KUIAS/RIMS, Kyoto University, Kyoto, 606-8305, Japan\\
\texttt{koji.fujiwara@oist.jp; kfujiwara@math.kyoto-u.ac.jp}

\end{document}